\documentclass{ws-jaa}
\begin{document}

\markboth{Stephen  D. Cohen and Anju Gupta}
{Primitive Element Pairs with a Prescribed Trace in the Quartic Extension of a  Finite Field}

%
\catchline{}{}{}{}{}
%

\title{Primitive Element Pairs with a Prescribed Trace in the Quartic Extension of a  Finite Field
}

\author{\footnotesize Stephen  D. Cohen\footnote {Postal address: \em{6 Bracken  Road, Portlethen, Aberdeen AB12 4TA, Scotland}}}

\address{Emeritus Professor of Number Theory, University of Glasgow\\
\email{Stephen.Cohen@glasgow.ac.uk}}
\author{Anju Gupta}

\address{Assistant Professor, Government College for Women,\\ Behal (Bhiwani), Haryana, 127021 India\\
\email{anjugju@gmail.com}}

\maketitle
\begin{history}
\received{(Day Month Year)}
\revised{(Day Month Year)}
\accepted{(Day Month Year)}
\comby{(xxxxxxxxx)}
\end{history}

\begin{abstract}
In this article, we give a largely self-contained proof  that the quartic extension $\mathbb{F}_{q^4}$ of the finite field $\mathbb{F}_q$ contains  a primitive element $\alpha $ such that the element $\alpha+\alpha^{-1}$ is also a  primitive element of ${\mathbb{F}_{q^4}},$ and $Tr_{\mathbb{F}_{q^4}|\mathbb{F}_{q}}(\alpha)=a$ for any prescribed $a \in \mathbb{F}_q$.

The corresponding result for finite field extensions of degrees exceeding 4  has already been established by Gupta, Sharma and Cohen.
\end{abstract}


\ccode{2010 Mathematics Subject Classification: 12E20; 11T23}

\section{Introduction}	
Let $\mathbb{F}_q$ denote the finite field of order $q=p^h$, where $p$ is a prime and $h$ is a  positive integer, and let $\mathbb{F}_{q^n}$ denotes an extension of $\mathbb{F}_q$ of degree $n.$ The multiplicative group $\mathbb{F}_q^*$ of $\mathbb{F}_q$ is cyclic and its generators   are called \textit{primitive elements} of $\mathbb{F}_q$.  The field ${\mathbb{F}_q}$ has $\phi(q-1)$ primitive elements, where $\phi$ is  Euler's phi-function.
Thus the proportion of non-zero elements of $\mathbb{F}_q$ that are primitive elements is $\theta(q-1)$, where for any positive integer $m$, $\theta(m) =\phi(m)/m$.

Over the years there have been a series of papers on the existence of primitive elements in the extension $\mathbb{F}_{q^n}$ of $\mathbb{F}_q$ of degree $n$.   Some of these are detailed in  \cite{GSC}.  For instance, an early result in \cite{Coh90} implies that, for every $n \geq 3$ and prime power $q$ and every $a$ in $\mathbb{F}_q$, there exists a primitive element in $\mathbb{F}_{q^n}$  with trace $a$ in $\mathbb{F}_q$, with the exception of the pair $(q,n)=(4,3)$.  Here, necessarily $n\geq 3$, since, if $n=2$, there is no primitive element with zero trace.
In \cite{GSC}, it was shown that given $n\geq 5$ and a prime power $q$, then, for any $a \in \mathbb{F}_q$ there exists a primitive element $\alpha \in \mathbb{F}_{q^n}$  with trace $a$ such that the related element $\alpha+\alpha^{-1}$ is also primitive.  In other words, in notation introduced  in \cite{GSC}, the pair $(q,n)$ lies in the set $\mathfrak{P}$.  The main purpose of the present paper is to extend this result to quartic extensions, i.e., $n=4$.  More precisely, we establish the following theorem.

\begin{theorem}\label{main}
	Let $q$ be a prime power and $n\geq 4$ a positive integer.  Then $(q,n) \in \mathfrak{P}$.
\end{theorem}

Given $q,n,a$ as in Theorem \ref{main}, denote by $N_a$ the number of primitive elements $\alpha \in \mathbb{F}_{q^n}$ with trace $a$ such that $\alpha+ \alpha^{-1}$ is also primitive.
In \cite{GSC}, to show that generally $N_a$ is positive, we first derived a sufficient condition  (the \emph{basic criterion}, BC) dependent on $\omega(q^n-1)$ (where, for a positive integer $m$, $\omega(m)$ is the number of distinct prime factors of $m$).  This was derived from  an expression for $N_a$ in terms of  character sums.   This establishes the result for large values of $\omega(q^n-1)$ (and so large values of $q^n$).  This was refined by a criterion, called the \emph{prime sieve criterion}(PSC) whereby most of the prime factors of $q^n-1$ are used as sieving primes leaving some (the smallest, including $2$ when $q$ is odd) as non-sieving primes.  This is useful in establishing the result for smaller values of $\omega(q^n-1)$ and $q^n$, even when the prime factorization of $q^n-1$ is not explicitly known, though, of course,  it is most effective when it is.   Finally, the complete result for $n \geq5$ was established by direct computation for  relatively few residual pairs $(q,n)$.

In this paper, we review the BC and PSC (from \cite{GSC}) and proceed to modify the PSC to produce the \emph{modified prime sieve criterion} (MPSC) by taking  one or more of the largest prime factors of $q^n-1$ out of the sieve and treating them differently.    We  then proceed  to the computational aspect of the investigation.

Of course, the question of whether a pair $(q,n) \in \mathfrak{P}$ is sensible for all integers $n \geq 3$ and, indeed, our theoretical discussion which follows is applied whenever $n\geq 3$.  When $n=3$, however, our criteria for membership of $\mathfrak{P}$ are too weak to provide a definitive answer to the existence question without a formidable degree of computation.  Hence, we content ourselves  in that case to the  the statement of a preliminary existence theorem for sufficiently large prime powers $q$ (Theorem \ref{cubic}) and defer further discussion to a further article.

\bigskip
\section{The basic criterion}\label{basicsec}
In this section, we give some  notation, basic definitions and results which are used in the subsequent sections.

Assume throughout $q$ is a given prime power and $n\geq3$ is an integer, and that $a \in \mathbb{F}_q$  is given.  An important notion is that of an \emph{ $e$-free} element of $\mathbb{F}_{q^n}$ which extends that of a primitive element.
For any $e|q^n - 1$, an element  $\alpha\in {\mathbb{F}_{q^n}^{*}}$  is said to be  \textit{$e$-free} if $\alpha = \gamma^d$, for any $d|e$ and $\gamma \in \mathbb{F}_{q^n}^{*}$  gives $d=1$. Hence an element of $\mathbb{F}_{q^n}$ is primitive if and only if it is $(q^n-1)$-free.  Observe that the definition is dependent only on the distinct primes in $e$, i.e., on $\mathrm{Rad}(e)$, the product of the distinct primes dividing $e$.

Now suppose that $e_1$ and $e_2$ are divisors of $q^n-1$.   Define $N_a(e_1,e_2)$ as the number of $\alpha \in \mathbb{F}_{q^n}^*$ that are $e_1$-free, have trace $a$ and are such that $\alpha+\alpha^{-1}$ is $e_2$-free.   The same number is obtained if $e_1, e_2$ are replaced by their radicals, in other words we can assume that $e_1$ and $e_2$ are square-free.  The particular case $N_a(q^n-1,q^n-1)$ yields $N_a$, defined  in Section 1, the number we are endeavouring to show positive for all values of $a$.   A key expression for $N_a(e_1,e_2)$ is given by (3) of \cite{GSC} as follows.

\begin{equation} \label{orange}
N_a(e_1,e_2)=
\frac{\theta(e_1)\theta(e_2)}{q}\sum_{d_1|e_1,d_2|e_2}\frac{\mu(d_1)\mu (d_2)}{\phi(d_1)\phi(d_2)}
\sum_{u\in \mathbb{F}_q}\psi_0(-au)\sum_{\chi_{d_1},\chi_{d_2}} S(\chi_{d_1},\chi_{d_2},u),
\end{equation}
where $\mu$ denotes the M\"{o}bius function and
$$S(\chi_{d_1}, \chi_{d_2},u)=\sum_{\alpha \in \mathbb{F}_{q^n}^*}\chi_{d_1}(\alpha)\chi_{d_2}(\alpha+\alpha^{-1})\hat{\psi}_0(u\alpha).$$
Here $\chi_{d}$  denotes  a multiplicative character of order $d$ and $\sum_{\chi_{d}}$ a sum over all such characters ($\phi(d)$ in number).
Moreover, $\psi_0$ is the canonical additive character on $\mathbb{F}_q$ and $\hat{\psi}_0$ is its extension as  an additive character of $\mathbb{F}_{q^n}$ defined by $\hat{\chi}_0(\alpha)= \chi_0(Tr_{\mathbb{F}_{q^n}|\mathbb{F}_q}(\alpha))$.
Define a constant $C_q$ by
\begin{equation}\label{Cq}
C_q=\begin{cases}
3, & \text{ if $q$ is odd},\\
2, & \text{if $q$ is even}.
\end{cases}
\end{equation}
From a theorem on  mixed character sums of rational functions (Lemma 2.2 of \cite{GSC}), we noted (Lemma 3.1 of \cite{GSC}) that, except when $u=0$ and $d_1=d_2=1$, that, in (\ref{orange}),

\begin{equation}\label{banana}
|S_(\chi_{d_1},\chi_{d_2},u)| \leq C_q q^{n/2}.
\end{equation}
In fact, better estimates than (\ref{banana})  hold when at least one (but not all)  of $u=0,\ d_1=1,\ d_2=1$ is true.  In particular,

$$|S(\chi_1,\chi_1,u)| \leq C_q, \quad u \neq 0.$$
This allows for the fact that $\chi_1(\alpha)\chi_1(\alpha^2+1)$ vanishes at the zeros of $\alpha^2+1$ =0 (if these lie in $\mathbb{F}_{q^n}$).
On the other hand,
$$S_(\chi_1,\chi_1,0) \geq q -C_q.$$
Thus,
\begin{equation} \label{peach}
\frac{1}{q}\sum\psi_0(-au)\sum_{\alpha \in \mathbb{F}_{q^n}^*}S(\chi_1,\chi_1,u) \geq q^{n-1} -C_q.
\end{equation}

Combining (\ref{banana}) and (\ref{peach}) in (\ref{orange}) and using the fact that each square-free pair  $(d_1,d_2) \neq(1,1)$ of divisors of $(e_1,e_2)$, respectively, contributes $-C_qq^{n/2}$ to the lower bound of $N_a(e_1,e_2)$, we deduce that
\begin{eqnarray}\label{almond}
N_a(e_1,e_2 )& \geq& \theta(e_1)\theta(e_2)\{q^{n-1}-C_q -C_q(W(e_1)W(e_2)-1)q^{n/2}\}\nonumber\\
& \geq&\theta(e_1)\theta(e_2)\{q^{n-1}-C_qW(e_1)W(e_2) q^{n/2 }\},
\end{eqnarray}
which corrects (6) of \cite{GSC}.  Here, for any integer positive $m$, $W(m)=2^{\omega(m)}$ is the number of square-free divisors of $m$.

The important case of (\ref{almond}) is when $e_1=e_2=k$, say.

\begin{theorem}\label{basic}
	Suppose $q$ is a prime power, $n\geq 3$ is an integer and $a$ is an arbitrary member of $\mathbb{F}_q$. Let $k$ be a divisor of $q^n-1$.   Then
	\begin{equation} \label{basicineq}
	N_a(k,k) \geq\theta^2(k)\{q^{n-1}-C_qW^2(k)q^{n/2}\}.
	\end{equation}
\end{theorem}
In particular, when $k=q^n-1$  in Theorem \ref{basic}, we obtain the basic criterion.
\begin{corollary}  Suppose $q$ is a prime power and $n \geq 3$ is an integer. Assume also that
	\begin{equation} \label{BC}
	q^{\frac{n}{2}-1} >C_qW^2(q^n-1).
	\end{equation}
	Then $(q,n) \in \mathfrak{P}$.
\end{corollary}

\bigskip
\section{The prime sieve}
We retain the notation of  Section \ref{basicsec}.  Let $k$ be a divisor of $q^n-1$  with $\mathrm{Rad}(k)<\mathrm{Rad}(q^n-1)$ and $P$ a coprime divisor of $q^n-1$ with $\mathrm{Rad}(P)=p_1\ldots p_r$ as a product of distinct primes, where $r \geq 1$. The prime sieve achieves a lower bound for $N_a(kP,kP)$ that is superior to an application of the basic criterion, provided an appropriate choice of $P$ (and so of $r$) is made.

As well as the lower bound (\ref{basicineq}), we require upper bounds for the absolute value of the differences $D:=N_a(k,kp)-\theta(p)N_a(k,k)$ and $N_a(kp,k)-\theta(p)N_a(k,k)$, where $p$ is any prime dividing $P$.  The significance of these differences is illustrated by the fact that, from (\ref{orange}),
\begin{eqnarray}\label{pineapple}
|D|&=&\frac{\theta(pk^2)}{q}\sum_{d_1|k,d_2|k}\frac{1}{\phi(pk^2)}\left|\sum_{u \in \mathbb{F}_q}\psi_0(-au)\sum_{\chi_{d_1}, \chi_{pd_2}}S(\chi_{pd_1}, \chi_{d_2},u)\right| \nonumber\\
&\leq&  \left(1- \frac{1}{p}\right)C_qW^2(k)q^{n/2},
\end{eqnarray}
by (\ref{banana}).
Similarly, the bound (\ref{pineapple}) applies to the other difference.

Next, we recall the sieving lemma, Lemma 3.3 of \cite{GSC}.

\begin{lemma}\label{sieve}
	With notation as above,
	\begin{equation}\label{sieveineq}
	N_a(kP,kP) \geq \sum_{i=1}^rN_a(p_ik,k)+\sum_{i=1}^rN(k,p_ik)-(2r-1)N_a(k,k).
	\end{equation}
\end{lemma}
Define
\begin{equation}\label{delta}
\delta = 1-2\sum_{i=1}^r\frac{1}{p_i}.
\end{equation}
Then the inequality (\ref{sieveineq}) can be written in  the long-winded but more useful form
\begin{multline*}
N_a(kP,kP)=\sum_{i=1}^r\left[N_a(kp_i,k)-\left(1-\frac{1}{p_i}\right)N_a(k,k)\right]\\
+\sum_{i=1}^r \left[N_a(k,kp_i)-\left(1-\frac{1}{p_i}\right)N_a(k,k)\right] \; +\;\delta N_a(k,k).
\end{multline*}
By (\ref{basicineq}) applied to yield a lower bound  to $N_a(k,k)$ and (\ref{pineapple}) applied to all the differences of the form $D$ (taken to be negative), we obtain (as in the proof of Theorem 3.4 of \cite{GSC}) a  better  inequality for $N_a(kP,kP)$.

\begin{theorem}\label{PSthm}
	With the above notation we have
	\begin{equation} \label{PSineq}
	N_a(kP,kP) \geq \theta^2(k)\{\delta q^{n-1} - C_q(2r-1+2\delta) q^{n/2} \}.
	\end{equation}
\end{theorem}

Observe that  Theorem \ref{PSthm} is vacuous if $\delta\leq 0$.  We take the case in which $kP=q^n-1$ and assume $\delta>0$ to obtain the {\em prime sieve criterion} (PSC).

\begin{theorem} \label{PSCthm}
	Suppose $q$ is a prime power and $n\geq 3$ is an integer.  Write $q^n-1$ in the form $kP$, where $k$ and $P$ are coprime, $\omega(k)=t$ and $\omega(P)=r \geq 1$.    Define $\delta$ by $ (\ref{delta})$ and assume $\delta>0$. Suppose that
	\begin{equation}\label{PSC}
	q^{\frac{n}{2}-1} > C_q2^{2t}\left(\frac{2r-1}{\delta}+2\right).
	\end{equation}
	Then $(q,n) \in \mathfrak{P}$.
\end{theorem}
The technique for applying Theorem \ref{PSCthm} is to include all the small primes in $q^n-1$ in $k$ (the \emph{core}) with as many as possible in $P$ so long as $\delta$ is positive.

\bigskip
\section{The modified prime sieve}

We continue to suppose $q$ is a given prime power, $n \geq 3$ and $a \in \mathbb{F}_q$. Write $\mathrm{Rad}(q^n-1)=kPL$, where $k$ (the \emph{core}) is the product of the $t$ smallest primes in $q^n-1$, where  $t\geq 1$, $P=p_1 \ldots p_r$ (the product of the \emph{main sieving primes}), $L= l_1\ldots l_s, s \geq 1$  (the product of the \emph{large primes}).

Define $\delta$ by (\ref{delta}), as before.   Also define $\varepsilon$ by
\begin{equation} \label{epsilon}
\varepsilon=\sum_{i=1}^{s}\frac{1}{l_i}.
\end{equation}

\begin{theorem}\label{MPSthm}  Let $q$  be a prime and $n\geq 3$ be an integer.  Write $\mathrm{Rad}(q^n-1)=kPL$, where $k$ is the core (with $\omega(k)=t$), $P$ the product of $r$  sieving primes, $L$ the product of $s$  large primes, as described above.  Assume that, with $\delta, \varepsilon $  as defined by $ (\ref{delta})$ and $(\ref{epsilon})$, respectively,
	$ \theta^2(k)\delta -2 \varepsilon $ is positive.  Suppose that
	
	\begin{equation}\label{MPSC}
	q^{\frac{n}{2}-1}>\frac{C_q\{\theta^2(k)(2r-1+2\delta)2^{2t} +(s-\varepsilon)\}}{\theta^2(k)\delta -2 \varepsilon}.
	\end{equation}
	
	Then $(q,n) \in \mathfrak{P}$.
\end{theorem}

\begin{proof}    As in Lemma \ref{sieve}
	\begin{equation}\label{sieve2}
	N_a  \geq N_a(kP,kP)+N_a(L,L)-N_a(1,1).
	\end{equation}
	
	In (\ref{sieve2}), use (\ref{PSineq}) as a lower bound for $N_a(kP,kP)$.

	Further,  by Lemma \ref{sieve},
	\begin{equation*}
	N_a(L,L)\geq \sum_{i=1}^{s}N_a(l_i, 1) +\sum_{i=1}^{s}N_a(1, l_i) - (2s - 1)N_a(1, 1),
	\end{equation*}
	which may be written as
	\begin{multline}\label{plum}
	N_a(L,L)-N_a(1,1)\geq \sum_{i=1}^{s}\{N_a(l_i, 1)-(1-\frac{1}{l_i})N_a(1,1)\} \\+\sum_{i=1}^{s}\{N_a(1, l_i)-(1-\frac{1}{l_i})N_a(1,1)\}  - 2\varepsilon N_a(1,1).
	\end{multline}
	
	Now $N_a(1,1)$ simply enumerates those non-zero $\alpha \in \mathbb{F}_{q^n}$ with trace $a$.  If $a \neq 0$, then, since $Tr(0)=0$, $N_a(1,1$)  can be expressed as
	
	$$N_a(1,1)=\frac{1}{q}\sum_{u\in \mathbb{F}_q}\psi(-au)\sum_{\alpha\in \mathbb{F}_{q^n}^*}\hat{\psi}_0(u\alpha)=\frac{1}{q}\sum_{u\in \mathbb{F}_q}\psi(-au)\sum_{\alpha\in \mathbb{F}_{q^n}}\hat{\psi}_0(u\alpha).$$
	Now,  $\sum_{\alpha \in \mathbb{F}_{q^n}}\hat{\psi}_0(u\alpha)=q^n$ if $u=0$ and is zero if $u\neq0$. Hence, when $a \neq 0$, $N_a(1,1)=q^{n-1}$ from which we deduce that $N_0(1,1)=q^{n-1}-1$.  So, certainly, for every $a \in \mathbb{F}_q$,
	$N_a(1,1) \leq q^{n-1}$.
	
	Moreover,  by a simple case of \cite{GSC}, Lemma 2.2, for any character $\chi_l$ of  prime order $l|L$, instead of (\ref{banana}), we can assume that
	
	\begin{equation}\label{fig}
	|S(\chi_l, \chi_1,u)|=\left|\sum_{\alpha \in \mathbb{F}_{q^n}^*}\chi_l(\alpha)\hat{\psi}_0(u \alpha)|\right|
	\leq q^{n/2},
	\end{equation}
	and\
	\begin{equation}\label{berry}
	|S(\chi_1, \chi_l,u)| = \left|\sum_{\alpha \in \mathbb{F}_{q^n}^*}\chi_l(\alpha+\alpha^{-1})\hat{\psi}_0(u \alpha)|\right|
	\leq(C_q-1) q^{n/2}.
	\end{equation}
	
	Thus,
	$$N_a(l,1)-\left(1-\frac{1}{l}\right)N_a(1,1)=\frac{\theta(l)}{q\phi(l)}\sum_{u \in \mathbb{F}_q}\sum_{\chi_l}S(\chi_l,\chi_1,u),$$
	whence, from (\ref{fig}),
	
	\begin{equation}\label{kiwi}
	\left|N_a(l,1)-\left(1-\frac{1}{l}\right)N_a(1,1)\right| \leq  \left(1-\frac{1}{l}\right) q^{\frac{n}{2}-1}.
	\end{equation}

	Similarly, from (\ref{berry}),
	\begin{equation}\label{melon}
	\left|N_a(1,l)-\left(1-\frac{1}{l}\right)N_a(1,1)\right| \leq (C_q-1) \left(1-\frac{1}{l}\right) q^{\frac{n}{2}-1}.
	\end{equation}
	
	Using (\ref{kiwi}) and (\ref{melon}) in (\ref{plum}) and the fact that $\sum_{i=1}^s(1-\frac{1}{l_i})=s-\varepsilon$, we deduce that
	\begin{equation}\label{mango}
	N_a(L,L)-N_a(1,1)\geq -C_q q^{n/2}(s- \varepsilon)-2\varepsilon q^{n-1}.
	\end{equation}
	
	The final step is to combine (\ref{PSineq}) and  (\ref{mango}) in (\ref{sieve2}) to yield the lower bound
	$$N_a\geq q^{n/2}\left\{\bigg[(\delta\theta^2(k)-2\varepsilon)q^{\frac{n}{2}-1}\bigg] -C_q \bigg[\theta^2(k)(2r-1+2\delta)2^{2t}+(s-\varepsilon)\bigg]\right\}.$$
	
	In particular, assume $\theta^2(k)\delta-2\varepsilon >0$ to deduce the { \em  modified prime sieve criterion} (or MPSC)
	(\ref{MPSC}).
	
\end{proof}
We observe that we recover the PSC from the MPSC by setting $s=\varepsilon=0$ and the BC from the PSC by setting $r=0, \delta=1$.  The MPSC can be effective when the PSC fails but not by much.

\bigskip
\section{Application of the criteria to extensions of any degree}\label{anydeg}

Continue to assume that $q$ is a prime power and $n\geq 3$.  Whereas, we could (after \cite{GSC}) restrict ourselves to degrees 3 and 4, in this section, we effectively deal with the general case by treating the ``worst" case ($n=3$).  From now on we abbreviate  $\omega(q^n-1)$ to $\omega$.

First we briefly describe the significance and value of the three criteria, BC (given by (\ref{BC})), PSC (given by(\ref{PSC})) and MPSC  (given by (\ref{MPSC})).  The BC  will be applied quite generally. For its use it requires only information about $\omega$ but, to be effective, we require to know that $\omega$ is explicitly sufficiently large.   However large this bound is, it is applied once only to give a cap on $\omega$ (and thereby on $q^n$) above which all pairs $(q,n)$  are guaranteed to be  in $\mathfrak{P}$.   Below this cap on $\omega$ (now treated as an upper bound), we can apply the PSC quite generally to a range of values of $\omega$.  It can be applied thus (and we shall do so in this section) even though we do not know the precise values of the sieving primes.  It is necessary only to make a good choice of $r$ so that $\delta >0$.   For  more specialised applications later in individual cases when the prime factorisation of $q^n-1$ is known, stronger conclusions can be derived.   For application of the MPSC, it is necessary to know the exact sieving primes and make a choice of one or more large primes.

In this section we shall make general application of the BC and the PSC.  For the former we give a specific upper bound for the multiplicative function $W(m)=2^{\omega(m)}$.

\begin{lemma}\label{g}
	Let   $m $ be a positive integer. Suppose $\omega(m) \geq 17922$. Then $W(m) < m^{1/16}.$
\end{lemma}
\begin{proof} Here let $p_i$ denote the $i$th prime.  In particular, $p_{17922}=199247$.  Set $$\displaystyle{P_{17922} = \prod_{i=1}^{17922} p_i \  (> 2.35\ldots \times 10^{86321})}.$$ Observe that for any $i \geq 17922$,  $p_i >2^{16}=65536$.  Evidently, by the multiplicativity of $W$,
	$$ \frac{2^{\omega(m)}}{m^{1/16}}\leq \frac{2^{17922}}{(P_{17922})^{1/16}} < 0.95$$
	and the result follows.
\end{proof}

We remark that Lemma \ref{g} implies that $\omega(m) < \frac{\log m}{16 \log 2}$ provided $m > 2.35\ldots \times 10^{86321}$.  Better explicit   bounds for $\omega(m)$  (more of the shape $\frac{\log m}{ \log \log m}$)  are, of course, available (e.g., \cite{Rob}), but Lemma \ref{g} is appropriate for our purposes.

\begin{lemma} \label{lime}
	Assume $q$ is a prime power and $n \geq 3$.  Further suppose that $\omega \geq 29$.  Then $(q,n) \in \mathfrak{P}$.

\end{lemma}

\begin{proof} First assume  $\omega\geq 17922$ (so that $q^n>2.35 \times 10^{86321})$.  Then by Lemma \ref{g}   and the BC, $(q,n)\in \mathfrak{P}$ if
	$$q^{\frac{3n}{8}-1}>3$$
	which holds for $n \geq 3$ (since $q > 3^8$ ).
	
	Although this cap on $\omega$ leaves an enormous number of prime powers this single use of the BC suffices to initiate the siveing process.  Henceforth we will apply the PSC.  Accordingly, we use  $R$ to denote the value on the right hand side of \eqref{PSC}.  Observe that $R$ does not depend explicitly on $q$ so that if $(\ref{PSC})$ holds we obtain an upper bound on the values of $q^n$ that remain as possible exceptions.
	
	Next assume that $365 \leq \omega\leq 17921$.
	
	Then,  with the notation of Theorem \ref{PSCthm}, take $k$ to be the factor of $q^n-1$  whose prime factors are the least $342$ primes dividing $q^n-1$. Thus $r \leq 17579$.
	Further, $\delta$ must be at least the value obtained when $r=17579$ and $P$ is the product of those primes from 2309 to 199211, inclusive.  Thus $\delta > 0.09434$. Hence $R<8.98\times10^{211}$.    Now \eqref{PSC} holds if  $q >R^{(2/(n-2))}$, i.e., if $q^n >R^{(2n/(n-2))}$, so certainly if $q^n>R^6$ (since $n \geq 3$),  i.e., $q^n>5.23\times 10^{1271}.$  
	If in fact $\omega \geq 430$ then $q^n> 3.2168\times10^{1273}.$ Hence the result holds.
	
	We next assume that $35 \leq\omega \leq 429$ 
	then with $\omega(k) =35$ and $r \leq 394$ the result holds 
	if  $\omega \geq 75$.
	
	Repeating the above process with $\omega(k)=11$ for $39\leq\omega\leq74$; with $\omega(k)=8$ for $31\leq\omega\leq38$; and with $\omega(k)=7$ for $29\leq\omega\leq30$ we see that the result holds for $\omega \geq 29.$

	This completes the proof.
\end{proof}

\smallskip

After Lemma \ref{lime} we may assume that $\omega \leq 28$.    If also $\omega \geq 6$, in  the PSC now take $\omega(k)=6$.   This yields  $\delta>0.024499$, 
and  $R<21592151$.   At this point the value of $n$ becomes significant.  In particular, we can suppose $q< 2.16\times 10^7$ if $n=4$ , whereas only $q<4.663\times 10^{14}$  if $n=3$.

\bigskip
\section{Quartic extensions: applying the PSC}

In this section suppose $n=4$.  Thus $\omega=\omega(q^4-1)$.  Also $ R$   denotes the value on the right hand side of \eqref{PSC} as in Section \ref{anydeg}.   For brevity, numbers in the calculations which follow are rounded down to integers or to 4 significant figures when decimals are involved, whereas the working was calculated to at least 12  significant figures.

As noted above we can conclude $(q,4) \in \mathfrak{P}$ if $q>R$, i.e., if $q^n>R^4$, i.e., if $q^4>2.173\times 10^{29}$.   If $\omega \geq 22$, however, simply by supposing that $q^4-1$ is at least the product of the first $22$ primes we see that  then $q^4>3.217\times 10^{30} $.

Hence we can assume that $\omega\leq 21$.   In this case, proceeding as in Lemma \ref{lime}, with $\omega(k)=5$ we obtain $\delta>0.02106$.
Thus, $R<4526364$ and the PSC is satisfied if $q^4>4.197\times 10^{26}$.     Now,  if $\omega \geq 20$ then $q^4>5.579\times 10^{26}$.   Hence again $(q,4) \in  \mathfrak{P}$ and we may suppose that $\omega \leq 19.$

We can play this game successfully down to the case  $\omega = 13$ . Then,  with $\omega(k)=4$ we get $\delta>0.1181$, $R<112037$ and  $q^4>3.042\times 10^{14},$ i.e., $q>4176.$ Hence we may assume that $4176<q<112038$. But, for these values of $q$, $\omega= 13$ precisely for one prime power, namely the prime $q=102829$; whence $q^4-1=2^4\cdot 3\cdot 5\cdot 7\cdot 11\cdot 13\cdot 17\cdot 19\cdot 37\cdot 89\cdot 113 \cdot 94441$. Then, $\omega(k) =4,\  (k=210), \  r=9$ yields  $\delta =0.2983, R=45289 <q$ and the PSC is satisfied.

Alerted by the need now to proceed cautiously, now take  $\omega = 12$.   Now, with  $\omega(k)=4  r =8$ we obtain  $\delta>0.1669$ and   $R<70545$. If  on the basis that $q^4-1$ is at least the product of the smallest $12$ primes we can deduce only that $q>1650.$         Hence we may assume that $1650<q<70546$.    Checking throughout this range of prime power values we find, however, that $\omega=12$ only for the 14 values of $q$ in the set  $\{20747,21013,25943,30103,38917,52571,53087,53129,53923,59753,\\ 60397,65963, 66347,66457\}$.  Check these in turn using the precise factorisation of $q^4-1$.   We illustrate with the two smallest prime powers.

So take $q=21013$, so that $q^4-1=2^4\cdot 3\cdot 5\cdot 7\cdot 17 \cdot 19\cdot 29\cdot 79\cdot 103\cdot 173\cdot 677$.  Here with $\omega(k) =3,\ (k=30), \  r=9$,  we have $\delta=0.2093$ and $R=15977<q$.  Similarly, the PSC succeeds for the 12 larger values of $q$.

On the other hand take $q=20747$.  Now $q^4-1=2^4\cdot 3\cdot 5\cdot 7\cdot 11\cdot 13\cdot 19\cdot 23\cdot 29\cdot 41\cdot 653\cdot 2273$.   For the best attempt to use the PSC, take $\omega(k)=4, \delta= 0.3504, R=34410$, which exceeds $q$.  Consequently,  the PSC fails and, alas,  $q=20747$ remains a possible exception.

In similar fashion, for $\omega \leq 11$, there are more possible exceptions, i.e., prime powers $q$ for which (\ref{PSC}) fails for any choice of $k, r$. For example, for $\omega =11$, there remain the 5 prime powers in the set $\{4217,  9043, 11131,23561, 38501\}$.

\begin{table}[!t]
	\centering
	\begin{tabular}{|c|c|p{10cm}|}
		\hline
		$\;  i$ & $|S_i|$& $ S_i$ \\
		
		\hline
		12& 1 & {20747}\\
		\hline
		11&4&{4217,  9043, 11131,23561}\\
		\hline
		10&18&\tiny{1597, 3541, 3739, 4027, 5641, 5741, \pmb{6089, 6397},  6469, 6733,7853, \pmb{8161, 8273}, 8581, 9283, 9547, \pmb{10009}, 6889}\\
		\hline
		9&57&\tiny{463, 659,853, 911,1123,1301,1331,1427,  1429, 1483,1607, 1721,1747, 1849,    1877, 1931, 2129, 2309, 2333, 2393,2437, 2551, 2621, 2633, 2707, 2801, 2843, 2861, 2939, 2969, 3011,3037, 3319, 3323, 3359, 3557, 3583,3613,\pmb{3761, 3863, 3947}, 4003,\pmb{ 4093}, 4159,\pmb{4229, 4421, 4423}, 4523, \pmb{
				4649, 4663, 4789, 4817}, 4831, 5039, \pmb{5179}, 5237, \pmb{6007}}\\
		\hline
		8&89&\tiny{307, 419,421,512, 599, 701,727, 743, 811, 827, 829, 857, 859, 919,967 1013, 1021, 1033, 1061, 1087, 1109,1217, 1223, 1231, 1259, 1277, 1289, 1291, 1303, 1321, 1327, 1373, 1409, 1481, 1487, 1553, 1559, 1567,
			1583, 1609, \pmb{1613, 1627}, 1637, 1723, 1777, 1789,  1847, 1861, 1871, 1973,2003, 2029, 2053, 2087, \pmb{2089}, 2111, 2141, 2143, 2197,\pmb{2213},\pmb{2209}, 2243,2267, 2287, 2311, 2339, 2423, \pmb{2477}, 2521,  2617,\pmb{2699},2729, 2927,3067, 3079, 3191,\pmb{3121, 3163,
				3331, 3389, 3433, 3571},3697, \pmb{3719, 3821, 3877},       \pmb{5279, 5851}, 6271}
		\\
		\hline
		7&84&\tiny{  83, 157, 173, 191, 211, 229, 233, 281, 293, 311, 313, 317,331, 337, 343,353, 373, 389, 401, 439, 443, 461, 467, 491, 499, 509, 523, 547, 557, 563, 571, 593, 601, 613, 617,
			619, 643, 647, 683, 691, 709, 729,733, 757, 761, 769, 787, 797, 839, 863, 883, \pmb{887}, 937, 941, 947, 953,961, \pmb{983}, 1009, 1019, 1049, 1063, 1091, 1093, \pmb{1097}, 1103, 1117, 1163, \pmb{1201}, 1279, 1399,
			1451, 1471, 1511, 1667,  1681, 1693, 1709, \pmb{1933}, 2113, \pmb{2281}, 2549, \pmb{2647, 2731}}
		\\
		\hline
		6&61&  \tiny{43, 47,64, 67, 73, 89, 103, 109, 113, 125,128, 131, 137, 139, 149, 151, 167,169, 179, 181, 197, 223, 227, 239, 241, 251, 263, 269, 277, 283, 347, 349, 359,361, 367, 379, 397, 409, 431,
			433, 457, 479, 487, 503, 521,529, 541, 569, 577, 587, 607, 631, \pmb{653, 661}, 673, 677, \pmb{739,841,} {881}, \pmb{991, 1051}}\\
		\hline
		5&31&\tiny{13, 23,27, 29, 31,32, 37, 41, 53, 59, 61, 71, 79,81, 97, 101, 107,121, 127, 163, 193, 199,243,\pmb{256,} 257, 271, 289,  \pmb{383}, 449 \pmb{641, 751}
		}\\
		\hline
		4&7&{8, 11, 16, 17, 19, 25}\\
		\hline
		3&4&{4,5,7,9}\\
		\hline
		2&2&{2,3}\\
		\hline
		
	\end{tabular}
	\caption{Possible exceptions to PSC with $\omega =i$}
\end{table}

\begin{lemma}\label{apricot}
	Suppose $q$ is a prime power.   Then $(q,4) \in \mathfrak{P}$  except for at most $358$  values of $q$.   The possible exceptions can be listed as a union of sets $\bigcup_{i=2}^{12}S_i$, where $S_i$ is the set of those with $\omega(k)=i$.
	The cardinality of  $S_i$ and its elements are as indicated in Table $1$.
	
\end{lemma}

\bigskip
\section{Quartic extensions: applying the MPSC}

We now attempt to apply the MPSC (\ref{MPSC}) to each of the possible exceptions listed in Lemma \ref{apricot}.  In essence, we transfer $s (\geq 1)$ of the sieving  primes in an application of the PSC to being large primes and thereby reduce the number of sieving primes  by $s$.  With  $r,s$ and $\delta, \varepsilon $ now as defined as in Theorem \ref{MPSthm} we  denote the right hand side of (\ref{MPSC}) by $R'$.

First we illustrate the improvement  with two examples.

First take $q=10009 \in S_{10}$. In fact,  $q^4-1 = 2^5\cdot 3^2\cdot 5\cdot 7\cdot11\cdot13\cdot 17\cdot 101 \cdot 139 \cdot 29173$.    Apply the PSC.  For the best result, take $k=30 (\omega(k)=3)$ and $r=7$.  Then $\delta =0.22671$ and $R=11393$.
Thus the PSC fails.  Now categorise  29173 so that $s=1$ and $r$ is reduced to 6.  This has little effect on $\delta$  (0.22678) but $R'=9925.04$ and the MPSC is satisfied.  Hence $(10009, 4) \in \mathfrak{P}$.
{\footnotesize{\begin{table}[!t]
			\center
			\begin{tabular}{|c|c|c|c|c|c|c||c|c|c|c|c|c|c|}
				\hline
				$q$ & $\omega$&$\omega_k$ &$r$& $s$ &$\delta$& $R'$ & $q$ & $\omega$&$ \omega_k$ &$r$& $s$ &$\delta$& $R'$ \\
				\hline
				\hline
				6397 &{\small 10} & 3& 6& 1 &{\small .4109}&{\small5834} & 2209 & 8 & 2 & 5 & 1&{\small.2209}&{\small 2181} \\
				\hline
				8161 & {\small}10 & 3& 6& 1 &{\small.3024}&{\small7728} & 2213 & 8 & 2 & 5& 1&{\small.2218}&{\small 2205} \\
				\hline
				8273 & {\small 10} & 3 & 6 & 1 &{\small .2771}&{\small 8223} & 2477 & 8 & 2 & 5 & 1&{\small .2081}& {\small 2311} \\
				\hline
				{\! \small 10009} &{\small  10}&3 & 6& 1 & {\small.2267}&{\small 9926}&   2699 & 8 & 2 & 5 & 1&{\small.1950}&{\small 2463} \\
				\hline
				3761 & 9&3& 5 & 1 &{\small .5320}& {\small 3758} & 3121 & 8 & 3 & 4& 1&{\small.5316}&{\small 2995} \\
				\hline
				3863 & 9 & 3 & 4 & 2 &{\small.4609}& {\small 3706}&3163 & 8 & 3 & 4 & 1&{\small .5144}& {\small 3079} \\
				\hline
				3947 & 9 & 3& 4 &2&{\small .5110} &{\small 3796} & 5279 & 8 & 3 & 4& 1&{\small .3096}&{\small 4861} \\
				\hline
				4093 & 9 & 3 & 5 & 1 &{\small .5265}&{\small  3756} &  5851 & 8 & 3 & 4 & 1&{\small .2733}&{\small  5455}\\
				
				\hline
				4229 & 9 & 3 & 5& 1&{\small .5321}&{\small 3799}& 887 & 7 & 2 & 4 & 1&{\small .4724}&{\small  877} \\
				\hline
				4421 & 9& 3 & 5 &1&{\small .5022}&{\small 3912}&  983 & 7 & 2 & 3& 2&{\small.3973}&{\small 927} \\
				\hline
				4423 & 9 & 3& 5 & 1 &{\small .4728}& {\small 4185} &1097 & 7 & 2 & 4 & 1&{\small .3987}& {\small 1012} \\
				\hline
				4649 & 9 & 3 & 5 & 1&{\small .4646}& {\small 4281}& 1201 & 7 & 2 & 4 & 1&{\small .3738}& {\small1095}\\

				\hline
				4663 & 9 & 3 & 5& 1&{\small.4272}&{\small 4548} &1933 & 7 & 2 & 4 & &{\small .2252}& {\small 1708} \\
				\hline
				4789 & 9 & 3 & 5 & 1&{\small .4470}&{\small 4508} &2281 & 7 & 2 & 4& 1&{\small .1967}&{\small 1942} \\
				
				\hline
				4817 & 9 & 3 & 5& 1&{\small .4548}&{\small 4339}&{\small 2647} & 7 & 2 & 4 & 1&{\small .1543}& {\small 2453} \\
				\hline
				5179 & 9 & 3 & 5 & 1&{\small .4351}& {\small 4461} & 2731 & 7 & 2 & 3 & 2&{\small .1575}& {\small 2401}\\
				\hline
				6007 & 9 & 3 & 5 & 1&{\small .3725}& {\small 5475} &  653 & 6 & 2 & 3 & 1&{\small .5693}& {\small 566} \\
				
				\hline
				1613 & 8 & 2 & 5& 1&{\small .3364}&{\small 1477}&  661 & 6 & 2 & 2& 2&{\small .4181}&{\small 655} \\
				\hline
				1627 & 8 & 2 & 5 & 1&{\small .3361}& {\small 1491}& 739 & 6 & 2 & 3 & 1&{\small .4971}&{\small 634} \\
				
				\hline
				2089 & 8 & 2 & 5 & 1&{\small .2409}& {\small 2048} &841 & 6 & 2 & 2 & 2&{\small.3142}&{\small 840.3}\\

				\hline
				3389 & 8 & 3 & 4& 1&{\small .5138}&{\small 3145} &991 & 6 & 2 & 3 & 1&{\small .3536}& {\small 852} \\
				\hline
				3433 & 8 & 3 & 4 & 1&{\small .5268}&{\small 3016} & 1051 & 6 & 2 & 3& 1&{\small.3066}&{\small 967} \\
				\hline
				3571 & 8 & 3 & 4 & 1&{\small .4488}& {\small 3473} &  256 & 5 & 2 & 1 & 2&{\small .8823}& {\small 180} \\
				
				\hline
				3719 & 8 & 3 & 4& 1&{\small .4821}&{\small 3267}& 383 & 5 & 2 & 1 & 2&{\small 0.6}& {\small 317}\\
				\hline
				3821 & 8 & 3 & 4 & 1&{\small .4323}& {\small 3591}&  641 & 5 & 2 & 1 & 2&{\small .5813}& {\small 391}\\
				\hline
				3877 & 8 & 3 & 4 & 1&{\small .4841}& {\small 3248} &751 & 5 & 2 & 2 & 1&{\small .5574}& {\small 403} \\
				\hline
				
			\end{tabular}
			\caption{Values of $q$ which satisfy the MPSC}
\end{table}}}

Next, take $q=3947 \in S_8$.  In fact, $q^4-1=2^4\cdot 3 \cdot 5\cdot 7\cdot 13\cdot 293 \cdot 409\cdot 1973$.  Applying the PSC with $k=30, r=7$ yields $\delta=0.0.50515$ and $R=4564.8 >q$ so that the PSC fails.  Applying the MPSC with the single large prime 1973 (thus $r=6, s=1$ and $\varepsilon=1/1973$) we obtain $\delta =0.5061, R'=3993.6>q$.  Therefore, with these parameters the MPSC also fails.   On the other hand, if we take both 1973 and 409 as large primes (thus $r=5, s=2$ and $\varepsilon =0.002951$) we obtain
$\delta =0.51106$ and $R'=3795.1 <q$.   Hence, by the MPSC, $(3947,4) \in \mathfrak{P}$.

In summary, for all the possible exceptions listed in Table 1 in bold type,  an application of the MPSC yields $(q,4) \in \mathfrak{P}$.
 The values of $q$ given in Table 2  (in which $\omega(k)$ is abbreviated to $\omega_k$)  satisfy the MPSC.

\begin{lemma}\label{guava}
	Suppose $q$ is a prime power.   Then $(q,4) \in \mathfrak{P}$  except for at most $304$  values of $q$.   The possible exceptions are those listed in Table $1$ that are \emph{not} in bold type.
	
\end{lemma}

\section{Direct verification for possibly exceptional prime powers}
For the exceptional values of $q$ in Table 1 that were not in bold type we used  computation to  verify the result.  In particular,    when  $\omega \leq 7$ we successfully
applied  gap4r8\cite{gap}  by means of  Algorithm 1 .  
For some prime powers with $\omega \geq 8$, however, Algorithm 1 takes too long.  In fact, all but five   of the possible   exceptions with $\omega \geq8$
(namely $512=2^8, \ 1331=11^3,\ 1849=43^2,\ 2197=13^3,\ 6889=83^2$)  are actually prime.
For {\em prime} values of $q$ with $\omega \geq 8$, we  used Algorithm 2. We give a brief description of the verification process.  For  a given prime $q$ choose $c$ to be least positive primitive root modulo $q$.   We searched for a primitive quartic polynomial of the form $f(x)=x^4-ax^3+bx+c$ for values of $b\geq 0$ in turn. (Here, necessarily, $c$ has to be primitive.)    Take $\alpha$ to be a root of $ f$ in $\mathbb{F}_{q^4}$.  Automatically $\alpha$ has trace $a$. We then used $\alpha$ to generate $\mathbb{F}_{q^4}$ over $\mathbb{F}_q$, formed $\alpha+1/\alpha$ and checked whether this element was also primitive.  If not we moved on to the next value of $b$ and repeated the procedure.   Eventually we were successful for every value of $a$.

\small{\begin{algorithm}[!t]
	\begin{algorithmic}[1]
		\caption{Check Whether $(q,n)\in \mathbb{F}_{q^n}$}
		\STATE Input: $q$ and $n$
		\STATE Output: Success if $(q,n)\in \mathfrak{P}$ and Failure if $(q,n)\not\in \mathfrak{P}$.
		\STATE Construct finite fields $\mathbb{F}_{q^n}$ and $\mathbb{F}_q$ and a primitive element $\gamma$ of $\mathbb{F}_{q^n}$.
		\STATE Construct a Set $P=\{\gamma^i\mid \i\leq q^n-1,(i,q^n-1)=1\}$, i.e., the set of all primitive elements of $\mathbb{F}_{q^n}$.
		\STATE Set $T=\{\}$
		\FOR {$a\in\mathbb{F}_q$}{
			
			\FOR {$b\in P$}
			
			\IF {Order$(b+b^{-1})=q^n-1$}
			\IF {Trace$(b)=a$}
			\STATE Store $a$ in $T$
			\STATE Break
			\ENDIF
			\ENDIF
			
			\ENDFOR
		}
		\ENDFOR
		\IF {Size$(T)=q$}
		\STATE Success
		\ELSE
		\STATE Failure
		\ENDIF
	\end{algorithmic}
	
\end{algorithm}}
By way of confirmation we applied a parallel method using the GaloisField packge of  Maple to  each $q$ in Table 1 with $\omega \geq 8$ at the same time keeping track of the largest value of $b$ needed for any $a$ and the value of the trace $a$ for which this value of $b$ was required.
For brevity full details are omitted but, we comment that, in practice, such examples are more prolific than our theoretical struggles suggest.   For the record, the largest value of $b$ required (amongst all  the prime values of $q$) was $654$ when $q=20747$, $c=5$  and $a=13548$.
The largest $b$ (as a fraction of $q$) was $b=212$ when $q=307$ and $a=251$.

The five non-prime powers $q$ were dealt with in a related way.  For example, when $q=2197=13^3$ we took $\mathbb{F}_{2197}=\mathbb{F}_{13}(z)$, where $z^3=2$.  Then $g=z+7$ is a primitive element of $\mathbb{F}_{2197}$.  Writing an arbitrary element of $\mathbb{F}_{2197}$ as $a=a_0+a_z+a_2z^2$ ($0\leq a_0,a_1,a_2 \leq 12$),  we sought primitive quartic polynomials of the form $x^4-ax^3+bx+g$, with $b=b_0+b_1z+b_2z^2$  ($0\leq b_0,b_1,b_2 \leq  12$) which had a root $\alpha$ for which $\alpha+1/\alpha$ was also primitive.  This was achieved in every case (even with $b_2=0$).

\begin{algorithm}[!t]
\begin{algorithmic}[1]
		\caption{Check Whether $(q,4)\in \mathbb{F}_{q^4}$}
		\STATE Input: $q$, a prime number
		\STATE Output: Success if $(q,4)\in \mathfrak{P}$ and Failure if $(q,4)\not\in \mathfrak{P}$.
		\STATE Construct finite fields $\mathbb{F}_{q^4}$, $\mathbb{F}_q$, the smallest primitive root $\gamma$ modulo ${q}$ and polynomial ring $\mathbb{F}_{q}[x]$.
		\STATE Set $S=\{\}$
		\FOR {$a\in\{0,1,2,\cdots q-1\}$}
		\FOR {$b\in\{1,2,\cdots q-1\}$}

		\IF {$f(x)=x^4-a*x^3+b*x+\gamma$ is a primitive polynomial over $\mathbb{F}_q$}
		\STATE Set $T=$ the set of roots of polynomial $f(x)$ in $\mathbb{F}_{q^4}$
		\FOR {$r\in T$}
		\IF {$r+r^{-1}$ is a primitive element of $\mathbb{F}_{q^4}$}
		\STATE Store $a$ in $S$
		\STATE Break
		\ENDIF
		
		\ENDFOR
		
		\ENDIF
		\STATE Break
		\ENDFOR

		\ENDFOR
		
		\IF {Size$(S)=q$}
		\STATE Success
		\ELSE
		\STATE Failure
		\ENDIF
	\end{algorithmic}
	
\end{algorithm}

\section{Cubic extensions}

For cubic extensions ($n=3$) we already conclude  (from Lemma \ref{lime}) that $(q,3) \in \mathfrak{P}$  whenever $\omega=\omega(q^3-1) \geq 29$.  Assuming $\omega \leq 28$ we can the exploit the fact that all prime factors of $q^2+q+1$ ($>3$) are congruent to $1$ modulo $6$ and use the PSC to establish the following preliminary existence theorem.

\begin{theorem}\label{cubic}
	Let $q$ be a prime power such that either $\omega(q^3-1) \geq 27$ or $q >3\times 10^{13}$.  Then $(q,3) \in \mathfrak{P}$.
\end{theorem}
Evidently Theorem \ref{cubic} can be improved.  It remains a formidable challenge to verify that  actually  only three pairs are not in $\mathfrak{P}$ (namely $(3,3), (4,3), (5,3)$) (as claimed in Conjecture 1 of \cite{GSC}).   We defer further discussion meantime.

\end{document}